\documentclass{amsart}

\usepackage{amssymb}
\usepackage{amsmath}
\usepackage{amsthm}
\usepackage{color}
\usepackage{hyperref,url}
\usepackage[shortlabels]{enumitem}
\usepackage{graphicx}
\usepackage{asymptote}

\newtheorem{theorem}{Theorem}[section]
\newtheorem{lemma}[theorem]{Lemma}
\newtheorem{corollary}[theorem]{Corollary}
\newtheorem{prop}[theorem]{Proposition}

\newtheorem{defn}[theorem]{Definition}

\theoremstyle{definition}
\newtheorem{remark}[theorem]{Remark}

\newcommand{\N}{\mathbb{N}}
\newcommand{\Z}{\mathbb{Z}}

\newcommand{\R}{\mathbb{R}}
\newcommand{\C}{\mathbb{C}}

\newcommand{\e}{\epsilon}

\newcommand{\lesim}{\lesssim}
\newcommand{\gesim}{\gtrsim}

\newcommand{\supp}{\mathrm{supp}}
\renewcommand{\Re}{\mathrm{Re}}

\title{A note on maximal operators for the Schr\"{o}dinger equation on $\mathbb{T}^1.$}

\author{Yuqiu Fu}
\address{Department of Mathematics, MIT,
Cambridge, MA 02139}
\email{yuqiufu@mit.edu}

\author{Kevin Ren}
\address{Department of Mathematics, Princeton University,
Princeton, NJ 08544}
\email{kevinren@princeton.edu}

\author{Haoyu Wang}
\address{Department of Mathematics, Yale University, New Haven, CT 06520}
\email{haoyu.wang@yale.edu}

\date{\today}

\begin{document}
\maketitle

\begin{abstract}
Motivated by the study of the maximal operator for the Schr\"{o}dinger equation on the one-dimensional torus $ \mathbb{T}^1 $, it is conjectured that for any complex sequence $ \{b_n\}_{n=1}^N $,
$$ \left\| \sup_{t\in [0,N^2]} \left|\sum_{n=1}^N b_n e \left(x\frac{n}{N} + t\frac{n^2}{N^2} \right) \right| \right\|_{L^4([0,N])} \leq C_\e N^{\e} N^{\frac{1}{2}} \|b_n\|_{\ell^2} $$
In this note, we show that if we replace the sequence $ \{\frac{n^2}{N^2}\}_{n=1}^N $ by an arbitrary sequence $ \{a_n\}_{n=1}^N $ with only some convex properties, then
$$ \left\| \sup_{t\in [0,N^2]} \left|\sum_{n=1}^N b_n e \left(x\frac{n}{N} + ta_n \right) \right| \right\|_{L^4([0,N])} \leq C_\e N^\e N^{\frac{7}{12}} \|b_n\|_{\ell^2}. $$
We further show that this bound is sharp up to a $C_\e N^\e$ factor.
\end{abstract}

\section{Introduction}

The maximal operator for the Schr\"{o}dinger equation on $\R^n$ has been studied extensively. In particular, the pointwise convergence of solutions to the linear Schr\"{o}dinger equation on $\R^n$ has been established for initial data in Sobolev spaces with sharp range of exponents \cite{du2019sharp}. 

For the Schr\"{o}dinger equation on the one-dimensional torus, some partial results on the $L^p$ estimates of the maximal operators were obtained. \cite{moyua2008bounds} proved that
\[ \| \sup_{0<t<1} |\sum_{n=1}^N a_n e(2\pi (nx + n^2t)) | \|_{L^6([0,1])} \leq C_\e N^{1/3 + \e} (\sum_{n=1}^N |a_n|^2)^{1/2},\]
which is sharp up to a $C_\e N^\e$ factor. Some estimates for higher dimensional tori were obtained in \cite{wang2019pointwise,compaan2021pointwise}.

It is conjectured that for every $\e>0,$ there exists a constant $C_\e>0$ such that
\begin{equation}\label{1}
    \left\| \sup_{t\in [0,N^2]} \left|\sum_{n=1}^N b_n e \left(x\frac{n}{N} + t\frac{n^2}{N^2} \right) \right| \right\|_{L^4([0,N])} \leq C_\e N^{\e} N^{\frac{1}{2}} \|b_n\|_{\ell^2},
\end{equation}
where $b_n$ are complex numbers. To recall,  $ e(x):=e^{2\pi i x}.$ 
It is shown in \cite{barron20204} that \eqref{1} holds when $b_n = 1$ for every $n.$ 


Similarly we could consider the maximal operator in the $x$ direction, and it seems plausible to conjecture that for every $\e>0,$ there exists a constant $C_\e$ such that 
\begin{equation}\label{2}
    \left\| \sup_{x\in [0,N]} \left|\sum_{n=1}^N b_n e \left(x\frac{n}{N} + t\frac{n^2}{N^2} \right) \right| \right\|_{L^4([0,N^2])} \leq C_\e N^{\e} N^{\frac{1}{2}} \|b_n\|_{\ell^2},
\end{equation}
where $b_n$ are complex numbers.

In this paper we describe some difficulties of this problem by showing both \eqref{1} and \eqref{2} cannot hold when we replace $\{\frac{n^2}{N^2}\}_{n=1}^N$ by some uniformly convex sequences, whose definition we now give.

\begin{defn}
We say a sequence $\{a_n\}_{n=1}^N$ (with parameter $N$) is uniformly convex if it satisfies
\begin{equation}\label{3}
    a_{n+1}-a_n \in \left[\frac{1}{4N}, \frac{4}{N}\right]
\end{equation}
\begin{equation}\label{4}
    (a_{n+2}-a_{n+1})-(a_{n+1} - a_n) \in \left[\frac{1}{4N^2}, \frac{4}{N^2}  \right].
\end{equation}
\end{defn}
This definition coincides with the definition of generalized Dirichlet sequences in \cite{fu2021decoupling}. We use the name of uniformly convex sequence in this paper in hope of better clarity in this context. 

We will construct examples of uniformly convex sequences to show the following two theorems.

\begin{theorem}\label{L4xLinftThm}
There exist uniformly convex sequences (for every parameter $N\geq 10$) $\{a_n\}_{n=1}^N,$ and complex sequences $\{b_n\}_{n=1}^{N}$ such that 
\begin{equation}\label{5}
    \left\| \sup_{t\in [0,N^2]} \left|\sum_{n=1}^N b_n e \left(x\frac{n}{N} + ta_n \right) \right| \right\|_{L^4([0,N])} \geq C N^{\frac{7}{12}} \|b_n\|_{\ell^2}
\end{equation}
for every $N\geq 10.$

Moreover, for every $\e>0$ there exists $C_\e>0$ such that for any uniformly convex sequence $\{a_n\}_{n=1}^N,$ and any complex sequence $\{b_n\}_{n=1}^N$ we have 
\begin{equation}\label{12}
    \left\| \sup_{t\in [0,N^2]} \left|\sum_{n=1}^N b_n e \left(x\frac{n}{N} + ta_n \right) \right| \right\|_{L^4([0,N])} \leq C_\e N^\e N^{\frac{7}{12}} \|b_n\|_{\ell^2}.
\end{equation}
\end{theorem}

\begin{theorem}\label{L4tLinfxThm}
There exist uniformly convex sequences (for every parameter $N\geq 10$) $\{a_n\}_{n=1}^N,$ and complex sequences $\{b_n\}_{n=1}^{N}$ such that 
\begin{equation}\label{11}
    \left\| \sup_{x\in [0,N]} \left|\sum_{n=1}^N b_n e \left(x\frac{n}{N} + ta_n \right) \right| \right\|_{L^4([0,N^2])} \geq C N^{\frac{5}{8}} \|b_n\|_{\ell^2}
\end{equation}
for every $N\geq 10.$

Moreover, for every $\e>0$ there exists $C_\e>0$ such that for any uniformly convex sequence $\{a_n\}_{n=1}^N,$ and any complex sequence $\{b_n\}_{n=1}^N$ we have 
\begin{equation}\label{n130}
    \left\| \sup_{x\in [0,N]} \left|\sum_{n=1}^N b_n e \left(x\frac{n}{N} + ta_n \right) \right| \right\|_{L^4([0,N^2])} \leq C_\e N^\e N^{\frac{2}{3}} \|b_n\|_{\ell^2}.
\end{equation}

\end{theorem}

The proof of the first parts of the theorems above relies on a family of examples of uniformly convex sequences having large intersections with an arithmetic progression. 

\begin{theorem}\label{SeqConstrThm}
   For every $\alpha \in [1/2,2],$ there exists a uniformly convex sequence $\{a_n\}_{n=1}^N$ for every $N\geq 10$ such that $|\{a_n\} \cap N^{-\alpha}\Z| \gtrsim N^{\frac{\alpha+1}{3}}.$ 
   
   For every $\alpha \in [0,1/2],$ there exists a uniformly convex sequence $\{a_n\}_{n=1}^N$ for every $N\geq 10$ such that $|\{a_n\} \cap N^{-\alpha} \Z| \gtrsim N^{\alpha}.$
   
   Conversely, for every $\e>0$ there exists $C_\e>0$ such that for every uniformly convex sequence $\{a_n\}_{n=1}^N$  and every $\alpha \in [0,2],$ we have
   \begin{equation}\label{201}
       |\{a_n\}_{n=1}^N \cap N^{-\alpha}\Z | \leq 
   \begin{cases}
   C_\e N^\e N^{\frac{\alpha + 1}{3}} & \text{ if } \alpha \in [\frac{1}{2}, 2], \\
   100N^{\alpha} & \text{ if } \alpha \in [0,\frac{1}{2}]
   \end{cases}.
   \end{equation}
\end{theorem}

The inequality \eqref{201} is Corollary 4.7 of \cite{fu2021decoupling} (when $\theta = 1$). The construction of example part of Theorem \ref{SeqConstrThm} will be proved by Theorem \ref{SeqConstrThm2} below. So the examples that we construct show that \eqref{201} is sharp up to a $C_\e N^\e$ factor. 

\vspace{.1in}
\noindent
{\bf Notation.} For two quantities $A,B,$ we use $A\lesssim B$ to denote the statement $A\leq CB$ for some constant $C.$ $A\lesssim_q B$ will denote the statement that $A\leq C_q B$ for some constant $C_q$ depending on $q.$ $A\sim B$ means $A\lesssim B$ and $B\lesssim A,$ and similarly $A\sim_q B$ means $A\lesssim_q B$ and $B\lesssim_q A.$ For $ x \in \R $ and $ A \subset \R $, let $ d(x,A) $ to denote the distance of the point $ x $ to the set $ A $.

\vspace{.1in}
\noindent
{\bf Acknowledgements.} The authors would like to thank Larry Guth for helpful discussions.

\section{The $L^4_x L^\infty_t$ estimate}
In this section, we prove Theorem \ref{L4xLinftThm}.

\subsection{Proof of \eqref{5}}
Let $\{c_n\}_{n=1}^N$ be a uniformly convex sequence from Theorem \ref{SeqConstrThm} with the property that $|\{c_n\}_{n=1}^N \cap N^{-1}\Z| \gtrsim N^{2/3}.$ 
Let $\{a_n\}_{n=1}^N$ be the sequence with $a_n = c_n - \frac{n}{N^2}.$
Then $\{a_n\}_{n=1}^N$ is a uniformly convex sequence because the sequence $\{ -\frac{n}{N^2}\}_{n=1}^N$ is an arithmetic progression with common difference $-\frac{1}{N^2}$.

Let 
\[f(x,t) = \sum_{1\leq n\leq N: c_n \in N^{-1}\Z} e(\frac{n}{N} x + a_n t),\]
that is, in \eqref{5} we are taking $b_n = 1$ if $Nc_n \in \Z$ and $b_n = 0$ otherwise.
We check that when $c_n \in N^{-1}\Z,$ we have
\[ (\frac{n}{N}, a_n) \cdot (j, j N) \in \Z \]
for every $1\leq j \leq N$ with $j\in \Z.$ Indeed when $c_n \in N^{-1}\Z,$
\begin{align} \label{n112}
(\frac{n}{N}, a_n) \cdot (j, jN) & = \frac{jn}{N} + jN a_n \\ \nonumber
& = \frac{jn}{N} + jN (c_n - \frac{n}{N^2})  \\  \nonumber
& = \frac{jn}{N} + jN c_n - \frac{jn}{N} \\ \nonumber
& = jN c_n \\ \nonumber
&\in \Z.
\end{align}
Therefore we have shown that 
\begin{equation}\label{n111}
    f(j, jN) = \#\{ n\in \N: 1\leq n \leq N, c_n \in N^{-1}\Z \} \gtrsim N^{2/3}.
\end{equation}
The calculation in \eqref{n112} in fact shows that for every $ 1 \leq j \leq N $ with $ j \in \Z $, if $|(x,t) - (j,jN) | \leq \frac{1}{100}$ we have
\begin{equation}\label{n113}
    |f((x,t))| \gtrsim N^{2/3}.
\end{equation}
Indeed, for every $1\leq n\leq N$ with $c_n \in N^{-1}\Z,$ if $|(j,jN)-(x,t)| \leq \frac{1}{100}$,
\begin{align*}
    d \left((\frac{n}{N}, a_n) \cdot (x, t), \Z \right) &\leq  d \left((\frac{n}{N}, a_n) \cdot (j, jN), \Z \right) + \left|(\frac{n}{N}, a_n)\right| |(j,jN)-(x,t)| \\
    &\leq 0 + 2 |(j,jN)-(x,t)|\\ 
    &\leq \frac{1}{50}.
\end{align*}
Therefore we have when $|(j,jN)-(x,t)| \leq \frac{1}{100}$,
\[ \Re(f(x,t)) \gtrsim \frac{1}{100} \#\{ n\in \N: 1\leq n \leq N, c_n \in N^{-1}\Z \} \gtrsim N^{2/3}, \]
which implies 
\begin{equation}\label{n115}
    |f(x,t)| \gtrsim N^{2/3}
\end{equation}
when $|(j,jN)-(x,t)| \leq \frac{1}{100}$ for $1\leq j \leq N.$

Because of \eqref{n115}, we have shown
\[ \left\{ (x,t): |(j,jN)-(x,t)| \leq \frac{1}{100}, 1\leq j \leq N  \right\} \subset \{(x,t) : |f(x,t)| \gtrsim N^{2/3} \},\]
which implies for every $1\leq j \leq N$, if $x\in [j-\frac{1}{100}, j+ \frac{1}{100}]$ then
\[ \sup_{t\in [0,N^2]} |f(x,t)| \gtrsim N^{2/3}. \]

Therefore,
\[ \| \sup_{t\in [0,N^2]} | f(x,t) | \|_{L^4([0,N])} \gtrsim N^{2/3} N^{1/4} = N^{7/12} \|b_n\|_{\ell^2}.  \]
This concludes the proof of \eqref{5}. 

\subsection{Proof of \eqref{12}}

Throughout this subsection, let $\{a_n\}_{n=1}^{N}$ be a uniformly convex sequence, and let $\{b_n\}_{n=1}^N$ be a sequence of complex numbers. 
The proof will utilize the $\ell^2L^6$ decoupling inequality of Bourgain and Demeter \cite{bourgain2015proof}.

\begin{theorem}[\cite{bourgain2015proof}]\label{BDdecoupling}
Suppose $f:[0,1] \rightarrow \R$ is $C^2$ and satisfies $|f'(x)| \leq 4$, $f''(x) \in [\frac{1}{4},4]$ for $x \in [0,1].$ Let $\Theta$ be a collection of disjoint subsets of $\R^2$ defined by $\{\theta: \theta = \{ (x, y)\in \R^2: x \in [\frac{j}{N}, \frac{j+1}{N}), |y - f(x)| \leq \frac{1}{N^2} \} \}.$ Suppose $f$ has Fourier transform supported on $\bigcup_{\theta \in \Theta} \theta,$ and let $f_\theta$ denote the Fourier projection $\int_\theta \hat{f}(\xi) e(x \cdot \xi) d\xi.$ For every $\e>0,$ there exists $C_\e$ such that 
\[ \| f \|_{L^6(\R^2)} \leq C_\e N^\e (\sum_{\theta \in \Theta} \|f_\theta\|_{L^6(\R^2)}^2 )^{1/2} \]
for every $f$ with $\supp \hat{f} \subset \bigcup_{\theta \in \Theta} \theta.$ 
\end{theorem}


Another ingredient we need is the locally constant property. First we set up some notations. The functions we define below are in $1$ or $2$ dimensions, that is, $n=1,2.$
Let $w(y)$ denote the function
\[ w(y) = \frac{1}{(1 + d(y,[-\frac{1}{2},\frac{1}{2}]^n))^{100}}. \]
So $w(y)$ is a smooth positive function which decays outside the $n$-dimensional cube $[-\frac{1}{2},\frac{1}{2}]^n$. For a unit cube $Y$ centered at $z\in \R^n,$ we write $w_{Y}(y)$ to denote the function $w(y-z),$ which decays outside the cube $Y.$ For $X,$ a subset of $\R^n,$ we let $\bigsqcup Y$ be the minimal covering of $X$ by lattice unit cubes, and write $w_{X} (y)$ to denote the function $\sum_{Y} w_Y(y).$ For a single point $x\in \R^n,$ we let $w_x(y)$ denote the function $w_{\{x\}} (y).$ We write $\| f \|_{L^p(w_X)}$ to denote the integral $(\int_{\R^n} |f(y)|^p w_{X}(y) dy)^{1/p}.$

The following locally constant property is well-known (see e.g. \cite{demeter2020fourier}).
\begin{prop}\label{locconstprop}
Let $n=1,2.$ Suppose $f$ has Fourier transform supported in $B_{10}(0).$ Then for $1\leq q < p \leq \infty,$
\[ \|f\|_{L^p(Y)} \lesssim_{p,q} (\int_{\R^n} |f|^q w_{Y})^{\frac{1}{q}}. \]
\end{prop}

Theorem \ref{BDdecoupling} has a local version (for example can be found as Theorem 5.1 in \cite{bourgain2017study}), which states that
\begin{equation}\label{n120}
    \| f \|_{L^6(w_{B_{N^2}})} \lesssim_\e N^\e (\sum_{\theta \in \Theta} \|f_\theta\|_{L^6(w_{B_{N^2}})}^2)^{1/2}
\end{equation}
for every $f:\R^2 \rightarrow \C$ with $\supp \hat{f} \subset \bigcup_{\theta \in \Theta} \theta.$ 

Now we are ready to prove \eqref{12}.
Denote the exponential sum $\sum_{n=1}^N b_n e(\frac{n}{N} x + a_n t)$ by $f(x,t).$ For every $\alpha>0,$ 
let $U_\alpha$ denote the level set 
\begin{equation}\label{eq:U_alpha}
U_\alpha := \{ (x,t) \in [0,N] \times [0,N^2]: |f(x,t)| \in [ \alpha/2, \alpha )  \}.
\end{equation}
The following proposition, which together with a dyadic pigeonholing argument will imply \eqref{12}.  

\begin{prop}\label{levelsettprop}
For every $\alpha>0$ we have 
\[ \alpha^4 | \pi_t U_\alpha|  \lesssim_\e C_\e N^\e N^{\frac{7}{3}} \|b_n\|_{\ell^2}^4  \]
where $\pi_t : \R^2 \rightarrow \R$ is the projection along the $t$-axis.
\end{prop}

\begin{proof}
For every $x\in \pi_t U_\alpha,$ let $t(x)\in [0,N^2]$ be such that 
\[ |f(x,t(x))| \geq \frac{\alpha}{2}. \]
By the definition of $\pi_t U_\alpha,$ we have
\begin{equation}\label{42601}
    \alpha^6 |\pi_t U_\alpha | \lesssim \int_{\pi_t U_\alpha}  |f(x,t(x))|^6 dx   . 
\end{equation}

Because of Proposition \ref{locconstprop}, we have 
\[ |f(x,t(x))| \lesssim (\int_{\R} |f(x,t)|^6 w_{t(x)}(t) dt)^{1/6}, \]
combining which with \eqref{42601} yields
\[ \alpha^6 |\pi_t U_\alpha | \lesssim \int_{\pi_t U_\alpha} \int_\R |f(x,t)|^6 w_{t(x)}(t) dx dt . \]
Noting the pointwise estimate
\[ 1_{\pi_t U_\alpha}(x) w_{t(x)}(t) \lesssim w_{U_\alpha}(x,t), \]
we obtain
\[ \alpha^6 |\pi_t U_\alpha | \lesssim \int_{\R^2} |f(x,t)|^6 w_{U_\alpha}(x,t) dx dt \lesssim \int_{\R^2} |f(x,t)|^6 w_{[0,N]\times [0,N^2]}(x,t) dx dt. \]
Since $\sum_{j=1}^{N} w_{[(j-1)N, jN]\times [0,N^2]}(x,t) \lesssim w_{[0,N^2]\times [0,N^2]},$ and the exponential sum $f(x,t)$ is $N$-periodic in $x$, we have
\[ \int_{\R^2} |f(x,t)|^6 w_{[0,N]\times [0,N^2]}(x,t) dx dt \lesssim \frac{1}{N} \int_{\R^2} |f(x,t)|^6 w_{[0,N^2]\times [0,N^2]}(x,t) dx dt. \]
Since $f(x,t)$ has Fourier support on the parabola, by \eqref{n120} we have
\begin{align*}
    \int_{\R^2} |f(x,t)|^6 w_{[0,N^2]\times [0,N^2]}(x,t) dx dt & \lesssim_\e N^\e (\sum_{n=1}^N \| b_n e(\frac{n}{N}x + a_nt) \|_{L^6(w_{B_{N^2}})}^2)^3 \\
    & \lesssim  N^\e N^{4} \|b_n\|_{\ell_2}^6.
\end{align*}
Hence, we have
\begin{equation}\label{n121}
     \alpha^6 |\pi_t U_\alpha | \lesssim_\e N^\e N^3 \|b_n\|_{\ell_2}^6.
\end{equation}
On the other hand we have the trivial bound
\begin{equation}\label{n122}
    |\pi_t U_\alpha | \leq N. 
\end{equation}
Combining \eqref{n121} and \eqref{n122} we obtain
\[ \alpha^4 | \pi_t U_\alpha|  \lesssim_\e C_\e N^\e N^{\frac{7}{3}} \|b_n\|_{\ell^2}^4.  \]

\end{proof}

To conclude \eqref{12}, we combine Proposition \ref{levelsettprop} with a dyadic pigeonholing argument.

\begin{proof}[Proof of \eqref{12}]
  Without loss of generality we assume $\| b_n \|_{\ell^2} = 1.$ Then by the triangle inequality, $|f(x,t)| < N^{50}.$
  So we may write 
  \begin{multline}\label{n161}
      \left\| \sup_{t\in [0,N^2]} |f(x,t) | \right\|_{L^4([0,N])}^4\\
      \leq \sum_{\alpha \in [N^{-100},N^{100}], \, \alpha \text{ dyadic}} \left\| \sup_{t\in [0,N^2]} \left|\sum_{n=1}^N b_n e \left(x\frac{n}{N} + ta_n\right) \right| \right\|_{L^4(\pi_t U_\alpha)}^4 + E
  \end{multline}
  where 
  \[ E = \int_{x\in [0,N], \, \sup_{t\in [0,N^2]} |f(x,t)| \leq N^{-100}} \left(\sup_{t\in [0,N^2]} |f(x,t) |\right)^4 dx  \leq N N^{-400} \leq N^{-100}.  \]
  
  By Proposition \ref{levelsettprop} we have
  \[ \left\| \sup_{t\in [0,N^2]} \left|\sum_{n=1}^N b_n e \left(x\frac{n}{N} + ta_n\right) \right| \right\|_{L^4(\pi_t U_\alpha)}^4 \lesssim_\e N^\e N^{\frac{7}{3}} \]
  for every $\alpha>0.$
  
  Since the number of dyadic $\alpha$ in $[N^{-100}, N^{100}]$ is $O(\log N),$ we conclude from \eqref{n161} that
  \[ \left\| \sup_{t\in [0,N^2]} |f(x,t) | \right\|_{L^4([0,N])}^4 \lesssim_\e N^\e N^{\frac{7}{3}}. \]
\end{proof}

\section{The $L^4_t L^\infty_x$ estimate}
In this section, we prove Theorem \ref{L4tLinfxThm}.

\subsection{Proof of \eqref{11}}

Let $\{a_n\}_{n=1}^N$ be a uniformly convex sequence from Theorem \ref{SeqConstrThm} with the property that $|\{a_n\}_{n=1}^N \cap N^{-1/2}\Z| \gtrsim N^{1/2}.$  
Let 
$$ f(x,t) = \sum_{1\leq n\leq N: a_n \in N^{-1/2}\Z} e \left(\frac{n}{N} x + a_n t\right),$$
that is, in \eqref{11} we are taking $b_n = 1$ if $N^{1/2} a_n \in \Z$ and $b_n = 0$ otherwise.
When $a_n \in N^{-1/2}\Z$ we have 
\[ (\frac{n}{N}, a_n) \cdot (0, j N^{1/2}) \in \Z \]
for every $1\leq j \leq N^2$ with $j\in \Z.$ 
Therefore 
\begin{equation}\label{n151}
    f(0, jN^{1/2}) = \#\{ n\in \N: 1\leq n \leq N, a_n \in N^{-1/2}\Z \} \gtrsim N^{1/2}.
\end{equation}
In fact, for every $1 \leq j \leq N^{3/2}$, when $|(x,t) - (0,jN^{1/2}) | \leq \frac{1}{100}$ we have
\begin{equation}\label{n153}
    |f((x,t))| \gtrsim N^{1/2}.
\end{equation}
Indeed, for every $1\leq n\leq N$ with $a_n \in N^{-1/2}\Z,$
\begin{align*}
    d \left((\frac{n}{N}, a_n) \cdot (x, t), \Z \right) &\leq  d \left((\frac{n}{N}, a_n) \cdot (0, jN), \Z \right) + \left|(\frac{n}{N}, a_n)\right| |(0,jN^{1/2})-(x,t)| \\
    &\leq 0 + 2 |(0,jN^{1/2})-(x,t)|\\
    &\leq \frac{1}{50}
\end{align*}
if $|(0,jN^{1/2})-(x,t)| \leq \frac{1}{100}.$  Therefore we have when $|(0,jN^{1/2})-(x,t)| \leq \frac{1}{100}$,
\[ \Re(f(x,t)) \gtrsim \frac{1}{100} \#\{ n\in \N: 1\leq n \leq N, a_n \in N^{-1/2}\Z \} \gtrsim N^{1/2}, \]
which implies 
\begin{equation}\label{n155}
    |f(x,t)| \gtrsim N^{1/2}
\end{equation}
when $|(0,jN^{1/2})-(x,t)| \leq \frac{1}{100}$ for $1\leq j \leq N^{3/2}.$

Because of \eqref{n155}, we have shown
\[ \left\{ (x,t): |(0,jN^{1/2})-(x,t)| \leq \frac{1}{100}, 1\leq j \leq N^{3/2}  \right\} \subset \{(x,t) : |f(x,t)| \gtrsim N^{1/2} \}\]
which implies 
\[ \sup_{x\in [0,N]} |f(x,t)| \gtrsim N^{1/2} \]
if $t\in [jN^{1/2}-\frac{1}{100}, jN^{1/2}+ \frac{1}{100}]$ for every $1\leq j \leq N^{3/2}.$

Therefore,
\[ \left\| \sup_{x\in [0,N]} | f(x,t) | \right\|_{L^4([0,N^2])} \gtrsim N^{1/2} N^{3/8} = N^{5/8} \|b_n\|_{\ell^2}.  \]
This concludes the proof of \eqref{11}.

\subsection{Proof of \eqref{n130}}
Recall $ U_\alpha $ is the level set of $ f(x,t) $ defined in \eqref{eq:U_alpha}.
\begin{prop}\label{prop:LevelSet_x}
For every $\alpha>0$ we have 
\[ \alpha^4 | \pi_x U_\alpha|  \lesssim_\e C_\e N^\e N^{\frac{8}{3}} \|b_n\|_{\ell^2}^4  \]
where $\pi_x : \R^2 \rightarrow \R$ is the projection long the $x$-axis.
\end{prop}

\begin{proof}
For every $t\in \pi_x U_\alpha,$ let $x(t)\in [0,N]$ be such that 
\[ |f(x(t),t)| \geq \frac{\alpha}{2}. \]
By the definition of $\pi_x U_\alpha,$ we have
\begin{equation}\label{426011}
    \alpha^6 |\pi_x U_\alpha | \lesssim \int_{\pi_x U_\alpha}  |f(x(t),t)|^6 dt   . 
\end{equation}

Because of Proposition \ref{locconstprop}, we have 
\[ |f(x(t),t)| \lesssim (\int_{\R} |f(x,t)|^6 w_{x(t)}(x) dx)^{1/6}, \]
combining which with \eqref{426011} yields
\[ \alpha^6 |\pi_x U_\alpha | \lesssim \int_{\pi_x U_\alpha} \int_\R |f(x,t)|^6 w_{x(t)}(x) dxdt . \]
Noting the pointwise estimate
\[ 1_{\pi_xU_\alpha}(t) w_{x(t)}(x) \lesssim w_{U_\alpha}(x,t), \]
we obtain
\[ \alpha^6 |\pi_x U_\alpha | \lesssim \int_{\R^2} |f(x,t)|^6 w_{U_\alpha}(x,t) dx dt \lesssim \int_{\R^2} |f(x,t)|^6 w_{[0,N]\times [0,N^2]}(x,t) dx dt. \]
Since $\sum_{j=1}^{N} w_{[(j-1)N, jN]\times [0,N^2]}(x,t) \lesssim w_{[0,N^2]\times [0,N^2]} (x,t),$ and the exponential sum $f(x,t)$ is $N$-periodic in $x$, we have
\[ \int_{\R^2} |f(x,t)|^6 w_{[0,N]\times [0,N^2]}(x,t) dx dt \lesssim \frac{1}{N} \int_{\R^2} |f(x,t)|^6 w_{[0,N^2]\times [0,N^2]}(x,t) dx dt. \]
Since $f(x,t)$ has Fourier support on the parabola, by \eqref{n120} we have
\[ \int_{\R^2} |f(x,t)|^6 w_{[0,N^2]\times [0,N^2]}(x,t) dx dt \lesssim_\e C_\e N^\e N^{4} \|b_n\|_{\ell_2}^6.  \]
Hence, we have
\begin{equation}\label{n1211}
     \alpha^6 |\pi_x U_\alpha | \lesssim_\e N^\e N^3 \|b_n\|_{\ell_2}^6.
\end{equation}
On the other hand we have the trivial bound
\begin{equation}\label{n1222}
    |\pi_x U_\alpha | \leq N^2. 
\end{equation}
Combining \eqref{n1211} and \eqref{n1222} we obtain
\[ \alpha^4 | \pi_x U_\alpha|  \lesssim_\e C_\e N^\e N^{\frac{8}{3}} \|b_n\|_{\ell^2}^4  \]
\end{proof}


\begin{proof}[Proof of \eqref{n130}]
Without loss of generality we assume $ \|b_n\|_{\ell^2} = 1 $. Similarly as in the proof of \eqref{12}, Proposition \ref{prop:LevelSet_x} yields
\begin{align*}
\left\| \sup_{x \in [0,N]} |f(x,t)| \right\|_{L^4([0,N^2])}^4
&\leq \sum_{\substack{\alpha \in [N^{-100},N^{100}],\\ \alpha \  \text{dyadic}}} \left\| \sup_{x \in [0,N]} |f(x,t)| \right\|_{L^4(\pi_x U_\alpha)}^4 + N^{-100}\\
&\leq C_\e (\log N) N^\epsilon N^{\frac{8}{3}} + N^{-100}\\
&\lesssim_{\epsilon} N^\epsilon N^{\frac{8}{3}},
\end{align*}
which completes the proof.
\end{proof}

\subsection{Remark}

Unlike Theorem \ref{L4xLinftThm}, we were not able to show a corresponding inequality where the example yielding \eqref{5} is  essentially sharp. However, if we consider the domain $(x,t) \in [0,N^2]^2$ instead of $[0,N] \times [0,N^2],$ and we are also allowed to perturb the ``periodic'' term $\frac{n}{N}$, then we do have a ``complete'' version of Theorem \ref{L4tLinfxThm}.

To be precise, we let $\{a_n\}_{n=1}^N$ be a uniformly convex sequence from Theorem \ref{SeqConstrThm} with the property that $|\{a_n\}_{n=1}^N \cap N^{-1}\Z| \gtrsim N^{2/3}.$ Let $\{v_n\}_{n=1}^N$ be a sequence in $\R^2$ defined by 
\begin{equation}\label{eq:v_n}
 v_n = (\frac{n}{N} - \frac{a_n}{N}, a_n).
\end{equation}




\begin{prop}\label{L4xLinfxLongThm}

For every $N\geq 10,$ there exists a complex sequence $\{b_n\}_{n=1}^{N}$ such that 
\begin{equation}\label{53001}
    \left\| \sup_{x\in [0,N^2]} \left|\sum_{n=1}^N b_n e((x,t) \cdot v_n) \right| \right\|_{L^4([0,N^2])} \geq C N^{\frac{5}{6}} \|b_n\|_{\ell^2}
\end{equation}
where $\{v_n\}_{n=1}^N$ is defined in \eqref{eq:v_n}.


\end{prop}

\begin{proof}
Let 
$$ f(x,t) := \sum_{1 \leq n \leq N, a_n \in N^{-1}\Z} e \left( (x,t) \cdot v_n \right), $$
i.e., take $ b_n=1 $ for $ n $ such that $ N a_n \in \Z $ and $ b_n=0 $ otherwise (so $\|b_n\|_{\ell^2} \sim N^{1/3}$).
For every $1\leq j \leq N^2$ with $j\in \Z$ we have
  \begin{align*}
      v_n \cdot (jN, j) & = (\frac{n}{N}-\frac{a_n}{N}, a_n) \cdot (jN,j) \\
      & = jn -ja_n + ja_n \\
      & = jn \\
      & \in \Z.
  \end{align*}
So 
  \[|f(jN,j)| = \#\{ n\in \N:1\leq n \leq N,\quad a_n \in N^{-1}\Z \} \gtrsim N^{2/3}.  \]
By the same argument in the proof of Theorem \ref{L4xLinftThm}, we actually have 
  \[ |f(x,t)| \gtrsim N^{2/3} \]
when $|(jN,j) - (x,t)| \leq \frac{1}{100}$ for $1\leq j \leq N^2.$
  
Therefore, we have 
  \[ \sup_{x\in [0,N^2]^2} |f(x,t)| \gtrsim N^{2/3} \]
if $t\in [j-\frac{1}{100}, j+\frac{1}{100}]$ for every $1\leq j\leq N^2.$ 
We conclude
  \[ \left\| \sup_{x\in [0,N^2]} |f(x,t)| \right\|_{L^4([0,N^2])} \geq C N^{\frac{2}{3}} N^{\frac{1}{2}} = C N^{\frac{5}{6}} \|b_n\|_{\ell^2}.\]
  
  
\end{proof}

\section{Construction of examples}
\begin{theorem}\label{thm:c2-interpolate}
Let $C \ge 1$ be a constant. Let $x_i, y_i, p_i$ be increasing sequences of real numbers. Let $\Delta x_i = x_{i+1} - x_i$, and define $\Delta y_i, \Delta p_i$ analogously. Suppose $\Delta p_i \sim_C \Delta x_i$ and $\frac{\Delta y_i}{\Delta x_i} - p_i \sim p_{i+1} - \frac{\Delta y_i}{\Delta x_i}$.
\begin{enumerate}[(a)]
    \item There exists a strictly convex $C^1$ function $f$ such that $f(x_i) = y_i$, $f'(x_i) = p_i$, and $f'$ is piecewise linear with slope $\sim_C 1$.
    
    \item Let $D$ be $\frac{\pi}{4}$ times the inf of slopes in (a). We can further modify $f$ to be a convex $C^2$ function with $f''(x_i) = D$ for all $i$ and $f'' (x) \sim_C 1$ for all $x \in (x_1, x_n)$.
\end{enumerate} 
\end{theorem}

\begin{proof}
It suffices to prove the case $n = 2$ for both (a) and (b); the general case is obtained by concatenation (note that the hypotheses ensure continuity at the boundary). We approach (a) first. Let $s = \frac{y_2 - y_1}{x_2 - x_1}$ and $c = \frac{s - p_1}{p_2 - s}$; then $c \sim 1$. Choose $x_0$ such that
\begin{equation*}
    \frac{x_2 - x_0}{x_0 - x_1} = c.
\end{equation*}
Choose $p_0$ such that $(x_1, p_2), (x_0, p_0), (x_2, p_1)$ are collinear; then $\frac{p_2 - p_0}{p_0 - p_1} = \frac{x_0 - x_1}{x_0 - x_2} = \frac{1}{c}$.
Let $f'$ be the piecewise linear function determined by lines from $(x_1, p_1)$ to $(x_0, p_0)$ and then to $(x_2, p_2)$. A calculation shows that the area under $f'$ from $x_1$ to $x_2$ is precisely $y_2 - y_1$, so $f(x_2) - f(x_1) = y_2 - y_1$. Thus, if we choose $f(x_1) = y_1$, then $f(x_2)$ must equal $y_2$.

To verify the slope condition, note that $p_0 - p_1 \sim p_2 - p_1$ and $x_0 - x_1 \sim x_2 - x_1$, so $\frac{p_0 - p_1}{x_0 - x_1} \sim \frac{p_2 - p_1}{x_2 - x_1} \sim 1$. Similarly, $\frac{p_2 - p_0}{x_2 - x_0} \sim 1$.

Finally, for (b), we observe the following fact, which follows from the Intermediate Value theorem: For every $y \in [0, \frac{\pi}{4}]$, there exists $x \in [\frac{\pi}{4}, \frac{\pi}{2}]$ with $y = x \cot x$.

Replace the linear function from $(x_1, p_1)$ to $(x_0, p_0)$ with a sinusoid $f'(x) = \frac{p_0 + p_1}{2} + \frac{p_0 - p_1}{2 \sin \alpha} \sin \frac{ \alpha(x - (x_0 + x_1)/2) }{(x_0 - x_1)/2}$, where $\alpha$ is chosen to satisfy $\alpha \cot \alpha = \frac{D (x_0 - x_1)}{p_0 - p_1}$ (which is allowed since the latter expression is $\le \frac{\pi}{4}$ by the definition of $D$). Make an analogous replacement for the linear function from $(x_0, p_0)$ to $(x_2, p_2)$. Then $f'$ becomes a $C^1$ function on $(x_1, x_2)$ with the boundary condition verified by $\lim_{x \to x_0^+} f''(x) = \lim_{x \to x_0^-} f''(x) = D$. Hence, $f$ is $C^2$. Also, replacing the linear function with the sinusoid doesn't change the area under the curve. Finally, $f''(x) = \frac{p_0 - p_1}{x_0 - x_1} \cdot \frac{\alpha}{\sin \alpha} \cos \frac{ \alpha(x - (x_0 + x_1)/2) }{(x_0 - x_1)/2}$. We have $f''(x) \le \frac{p_0 - p_1}{x_0 - x_1} \cdot \frac{\alpha}{\sin \alpha} \lesim 1$ (using $\frac{\alpha}{\sin \alpha} \le \frac{\pi}{2}$), and $f''(x) \ge \frac{p_0 - p_1}{x_0 - x_1} \cdot \frac{\alpha}{\sin \alpha} \cdot \cos \alpha = D$. In particular, we have $f''(x) > 0$, showing that $f$ is strictly convex.
\end{proof}

\begin{corollary}\label{cor:c2-interpolate}
If $\{ a_n \}$ is a $C$-uniformly convex sequence, then there exists a $C^2$ function $f$ passing through $(\frac{n}{N}, a_n)$ such that $f' \sim_C 1$ and $f'' \sim_C 1$.
\end{corollary}

\begin{proof}
We may define $a_0$ and $a_{N+1}$ such that $(a_n)_{n=0}^{N+1}$ is also $C$-uniformly convex, e.g. $a_0 = 2a_1 - a_2 + \frac{1}{N^2}$ and $a_{N+1} = 2a_N - a_{N-1} + \frac{1}{N^2}$. Apply Theorem \ref{thm:c2-interpolate}(b) to $x_i = \frac{i}{N}$, $y_i = a_n$, $p_i = \frac{N}{2} \cdot (a_{i+1} - a_{i-1})$. Then
\begin{equation*}
    \Delta p_i = \frac{N}{2} (\Delta a_{i+1} - \Delta a_{i-1}) = \frac{N}{2} (\Delta^2 a_i + \Delta^2 a_{i-1}) \sim_C \frac{1}{N} \sim \Delta x_i
\end{equation*}
and
\begin{gather*}
    \frac{\Delta y_i}{\Delta x_i} - p_i = N (a_{i+1} - a_i) - \frac{N}{2} \cdot (a_{i+1} - a_{i-1}) = \frac{N}{2} (a_{i+1} - 2a_i + a_{i-1}) \sim \frac{1}{N}, \\
    p_{i+1} - \frac{\Delta y_i}{\Delta x_i} = \frac{N}{2} (a_{i+2} - a_i) - N (a_{i+1} - a_i) = \frac{N}{2} (a_{i+2} - 2a_{i+1} + a_i) \sim \frac{1}{N}.
\end{gather*}
Thus, $\frac{\Delta y_i}{\Delta x_i} - p_i \sim p_{i+1} - \frac{\Delta y_i}{\Delta x_i}$, verifying the hypotheses of Theorem \ref{thm:c2-interpolate}.
\end{proof}

\begin{remark}
The interpolation of a convex sequence by a $ C^2 $ convex curve was considered in \cite{delbourgo1989shape,mulansky1998interpolation} etc. However, these results are mostly implicit statements of the existence. An explicit construction based on Bezier curve was mentioned in \cite{mathoverflow}. However, this Bezier cuve has zero second derivatives at the data points, and therefore it is not strictly convex. Our approach in Theorem \ref{thm:c2-interpolate} gives a simple construction and ensures strict convexity.
\end{remark}




\begin{lemma}
Let $x, y > 0$ be real numbers with $xy \ge 100$. The number of rational numbers $\frac{r}{s}$ in $[x, 2x]$ with denominator $s \le y$ is $\sim xy^2$.
\end{lemma}

\begin{proof}
Round up $y$ to the nearest integer.

To prove $\lesim xy^2$, we observe that the number of fractions in $[x, 2x]$ with denominator $k$ is $\le 1 + kx$, so the number of fractions with denominator $\le y$ is $\le y + xy^2 \lesim xy^2$. Now, we prove $\gesim xy^2$.

The number of fractions in $[0, 1]$ with denominator at most $n$ is $\frac{3}{\pi^2} n^2 + O(n \log n)$ (see for example \cite{fareycount}). This result extends to any interval $[m, m+1]$ (since if $x$ has denominator $n$, then so does $x + m$), and so the number of fractions in $[0, m]$ with denominator at most $n$ is $\sim \frac{3}{\pi^2} mn^2$. This proves the lemma for $x \ge 1$.

Let $F_n (m, x)$ be the number of fractions with denominator $n$ in $[mx, (m+1)x)$, and $G_n (m, x)$ be the number of fractions with denominator $\le n$ in $[mx, (m+1)x)$. We know $G_n (m, x) = \sum_{k=1}^n F_n (m, x)$. Furthermore, $F_n (m, x)$ equals the number of integers in $[mnx, (m+1)nx)$, so it equals $\lfloor nx \rfloor$ or $\lceil nx \rceil$. Thus, we have $|F_n (a, x) - F_n (b, x)| \le 1$, so $|G_n (a, x) - G_n (b, x)| \le n$ for all $a, b$. Let $s = \lceil \frac{1}{x} \rceil$ and $G_n (x) = \sum_{k=0}^{s-1} G_n (k, x)$; then, we have $|G_n (a, x) - \frac{1}{s} G_n (x)| \le \frac{1}{s} \sum_{k=0}^{s-1} |G_n (a, x) - G_n (k, x)| \le n$. Since $G_n (x) \ge \frac{3}{\pi^2} n^2$, we see that as long as $xy \ge 100$, we have $G_y (1, x) \gesim xy^2$.
\end{proof}

\begin{theorem}\label{SeqConstrThm2}
Choose $r_i$ to be the fractions in $[\frac{N^{\alpha-1}}{3}, \frac{2N^{\alpha-1}}{3}]$ with denominator $\le N^{(2-\alpha)/3}$. There are $\sim N^{(\alpha+1)/3}$ many fractions, and $r_{i+1} - r_i \gesim \frac{1}{N^{(4-2\alpha)/3}}$. Define $a_i = \frac{r_i}{N^\alpha}$, and now we will define $k_i, M_i$ as follows: For a given $i$, let $r_i = \frac{a}{b}$ and $r_{i+1} = \frac{c}{d}$, where we can choose $\Delta_i \le b, d \le 2\Delta_i$ for $\Delta_i := N^{2-\alpha} (r_{i+1} - r_i)$ (which is allowed since $\Delta_i \gesim N^{(2-\alpha)/3}$). Then let $M_i = a + c$ and $k_i = b + d$. Define the sequences $x_i = \frac{1}{N} \sum_{j=1}^i k_j$, $y_i = \frac{M_i}{N^\alpha}$, $p_i = N a_i$. Then there exists a convex $C^2$ function $f$ through $(x_i, y_i)$ with $f'(x_i) = p_i$ and $f''(x) \sim 1$.
\end{theorem}

\begin{proof}

Write $r_i = \frac{m_i}{n_i}$ in lowest terms. Then $r_{i+1} - r_i = \frac{m_{i+1} n_i - m_i n_{i+1}}{n_i n_{i+1}} \ge \frac{1}{n_i n_{i+1}} \gesim \frac{1}{N^{(4-2\alpha)/3}}$.

Now we check the conditions of Theorem \ref{thm:c2-interpolate}. First, $\Delta p_i = N(a_{i+1} - a_i) = \frac{r_{i+1} - r_i}{N^{\alpha-1}} \sim \frac{k_i}{N} = \Delta x_i$. Next,
\begin{gather*}
    \frac{M_i}{k_i N^\alpha} - a_i = \frac{1}{N^\alpha} \left( \frac{a+c}{b+d} - \frac{a}{b} \right) = \frac{1}{N^\alpha} \left( \frac{bc-ad}{b(b+d)} \right), \\
    a_{i+1} - \frac{M_i}{k_i N^\alpha} = \frac{1}{N^\alpha} \left( \frac{c}{d} - \frac{a+c}{b+d} \right) = \frac{1}{N^\alpha} \left( \frac{bc-ad}{d(b+d)} \right).
\end{gather*}
This implies $\frac{\Delta y_i}{\Delta x_i} - p_i \sim p_{i+1} - \frac{\Delta y_i}{\Delta x_i}$, so Theorem \ref{thm:c2-interpolate} applies.
\end{proof}

\begin{remark}
Our construction also gives a lower bound for the intersection of a short convex sequence with an arithmetic progression. An appropriate scaling of the short sequence yields a generalized Dirichlet sequence (see \cite[Definition 3.1]{fu2021decoupling}), which is a sequence $ \{b_n\} $ satisfying
$$ b_{n+1}-b_n  \in [\frac{1}{4N},\frac{4}{N}],\ \ \ (b_{n+2}-b_{n+1})-(b_{n+1}-b_n) \in [\frac{\theta}{4N^2},\frac{4\theta}{N^2}],\ \ \ 0<\theta\leq 1. $$
Consequently, we can extend Theorem \ref{SeqConstrThm} to generalized Dirichlet sequences with parameter $ \theta \not\equiv 1 $.

Specifically, for any $ 0<\theta(N)\leq 1 $ and $ \alpha \in [\tfrac{1}{2},2] $, there exists a generalized Dirichlet sequence $ \{b_n\}_{n=1}^N $ with parameter $ \theta $ such that
\begin{equation}\label{eq:Intersect_General}
\left| \{b_n\}_{n=1}^N \cap N^{-\alpha}\Z \right| \gtrsim \theta^{\frac{1}{3}} N^{\frac{1+\alpha}{3}}.
\end{equation}
To see this, for some fixed $ 0<\beta\leq 1 $, we restrict our attention to the first $ N^\beta $ terms of a convex sequence. Note that for the construction $ \{a_n\}_{n=1}^N $ in Theorem \ref{SeqConstrThm2}, the arithmetic progression is approximately equidistributed in the convex sequence. This yields that
\begin{equation}\label{eq:Intersect_Short}
\left| \{a_n\}_{n=1}^{N^{\beta}} \cap N^{-\alpha'}\Z \right| \gtrsim N^{\beta-1} N^{\frac{1+\alpha'}{3}},\ \ \  \tfrac{1}{2} \leq \alpha' \leq 2.
\end{equation}
Let $ \tilde{N}=N^\beta $, then $ \{N^{1-\beta}a_n\}_{n=1}^{\tilde{N}} $ is a generalized Dirichlet sequence of length $ \tilde{N} $ with parameter $ \theta=N^{\beta-1}=\tilde{N}^{1-\frac{1}{\beta}} $. Denote $ \tilde{\alpha}=\tfrac{\alpha'+\beta-1}{\beta} $, then rewriting \eqref{eq:Intersect_Short} gives
$$ \left| \{N^{1-\beta} a_n\}_{n=1}^{\tilde{N}} \cap \tilde{N}^{-\tilde{\alpha}}\Z \right| \gtrsim \tilde{N}^{\frac{\alpha'+3\beta-2}{3\beta}}  =  \theta^{\frac{1}{3}} \tilde{N}^{\frac{1+\tilde{\alpha}}{3}}, $$
which implies \eqref{eq:Intersect_General}.
\end{remark}

\bibliography{main.bib}
\bibliographystyle{alpha}

\end{document}